\documentclass{amsart}[11pt]
\usepackage{amsmath, amsfonts, amsthm, amssymb}
\usepackage{subfigure}
\usepackage{stmaryrd}
\usepackage{verbatim}
\usepackage{hyperref}
\usepackage{color}
\usepackage{ulem}
\usepackage[dvips]{graphicx}
\usepackage{enumerate}
\numberwithin{equation}{section}
\newtheorem{theorem}{Theorem}[section]

\newtheorem{lemma}[theorem]{Lemma}
\newtheorem{proposition}[theorem]{Proposition}

\theoremstyle{definition}
\newtheorem{definition}[theorem]{Definition}
\newtheorem{remark}[theorem]{Remark}

\newcommand{\ind}{1\hspace{-2.1mm}{1}}
\newcommand{\I}{\mathtt{i}}
\newcommand{\ff}{\mathrm{f}}
\newcommand{\g}{\mathrm{g}}
\newcommand{\RR}{\mathbb{R}}
\newcommand{\PP}{\mathbb{P}}
\newcommand{\EE}{\mathbb{E}}
\newcommand{\Dd}{\mathcal{D}}
\newcommand{\Cc}{\mathcal{C}}
\newcommand{\NN}{\mathbb{N}}
\newcommand{\MM}{\mathrm{M}}
\newcommand{\DdM}{\mathcal{D}_\MM}
\newcommand{\D}{\mathrm{d}}

\newcommand{\E}{\mathrm{e}}
\newcommand{\ZZ}{Z^{\MM}_{x,t}}
\newcommand{\QQ}{\mathbb{Q}^{\MM}_{x,t}}

\newcommand{\sgn}{\mathrm{sgn}}

\newcommand{\eps}{\varepsilon}
\newcommand{\be}{\begin{equation}}
\newcommand{\ee}{\end{equation}}

\def\equalDistrib{\,{\buildrel \Delta \over =}\,}

\begin{document}

\title{Two examples of non strictly convex large deviations}
\author{Stefano De Marco}
\address{CMAP, Ecole Polytechnique Paris}
\email{demarco@cmap.polytechnique.fr}
\author{Antoine Jacquier}
\address{Department of Mathematics, Imperial College London}
\email{a.jacquier@imperial.ac.uk}
\author{Patrick Roome}
\address{Department of Mathematics, Imperial College London}
\email{p.roome11@imperial.ac.uk}
\date{\today}
\thanks{
}
\thanks{AJ acknowledges financial support from the EPSRC First Grant EP/M008436/1.}

\begin{abstract}
We present two examples of a large deviations principle where the rate function is not strictly convex.
This is motivated by a model used in mathematical finance (the Heston model), 
and adds a new item to the zoology of non strictly convex large deviations.
For one of these examples, we show that the rate function of the Cram\'er-type of large deviations
coincides with that of the Freidlin-Wentzell when contraction principles are applied.
\end{abstract}

\maketitle

\section{Introduction}

The G\"artner-Ellis theorem is a key result in the theory of (finite-dimensional) large deviations.
Extending the results of Cram\'er~\cite{Cramer} 
for sequences of random variables not necessarily independent and identically distributed (iid), it provides a large deviations framework
solely based on the knowledge of the cumulant generating function (cgf) of the sequence.
The key assumptions are that the pointwise (rescaled) limit of these cgf satisfies some convexity property
and becomes steep at the boundaries of its effective domain;
this in turns implies that the rate function governing the large deviations, defined as the convex dual, 
is also convex.
Note that, by definition of the convex dual, essential smoothness of the limiting cgf 
implies strict convexity of the rate function.
When these assumptions are not met, large deviations 
(potentially with non strictly convex rate function) may or may not hold;
the classical example~\cite[Remark (d), page 45]{DZ93} is that of the sequence $(Z_n)_{n\in\mathbb{N}}$ 
distributed as exponential random variables with parameter~$n$.
It is immediate to see that 
$\Lambda(u):=\lim_{n\uparrow\infty}n^{-1}\log\EE(\E^{nu Z_n}) = 0$ if $u<1$ and is infinite otherwise.
Clearly, the function~$\Lambda$ is not essentially smooth (see Definition~\ref{def:EssSmooth}),
and the assumptions of the G\"artner-Ellis theorem, recalled in Appendix~\ref{app:LDP}, are violated;
nevertheless, a simple computation reveals that the conclusion of the latter still holds,
namely that a large deviations principle exists, with speed~$n$ and rate function 
$\Lambda^*(x):=\sup_{u}(ux-\Lambda(u)) = x$ if $x\geq 0$, and infinity otherwise;
clearly here, strict convexity of~$\Lambda^*$ does not hold.
Dembo and Zeitouni~\cite{DZ95} and Bryc and Dembo~\cite{BrycDembo}---in the context
of quadratic functionals of Gaussian processes---have proposed a way 
to bypass this absence of essential smoothness (of~$\Lambda$) / strict convexity (of~$\Lambda^*$) issue
by making the change of measure (key tool in the proof of the G\"artner-Ellis theorem)
time-dependent.
More recently, O'Brien~\cite{Brien} and Comman~\cite{Comman} have strengthened this theorem,
by partially relaxing the steepness and convexity assumptions.
In a general infinite-dimensional setting, Bryc's Theorem~\cite{Bryc} (see also~\cite[Chapter 4.4]{DZ93}), 
or `Inverse Varadhan's lemma', allows for large deviations with non convex rate functions. 
One of the hypotheses this theorem relies on is an exponential tightness requirement on the family of random variables under consideration, which is not always easy to verify.
However, several examples have been dug out which do not fall into this framework, such as in the setting of random walks with interface~\cite{Dupuis}, 
occupation measures of Markov chains~\cite{Jiang}, 
the on/off Weibull sojourn process~\cite{Duffy}, 
or $\mathrm{m}$-variate von Mises statistics~\cite{Eichel}.

Motivated by recent developments on large deviations in mathematical finance 
(see in particular~\cite{CDMD, DFJV, JR}, and the excellent review paper~\cite{Pham}), 
we study the small-time behaviour of the solution of the Feller stochastic differential equation 
(and an integral version of it) when the starting point is null. 
The absence of Lipschitz continuity of the diffusion coefficient and the degenerate starting condition 
make it not amenable to the classical Freidlin-Wentzell framework,
and the absence of strict convexity of the limiting moment generating function
violates the G\"artner-Ellis assumptions.
It turns out that a large deviations principle however holds, 
and one can furthermore reconcile the pathwise large deviations to the marginal (G\"artner-Ellis one)
by contraction.
We believe this provides a nice example of non-strictly-convex large deviations principle
in the context of continuous-time stochastic processes.
It also sheds light on the importance of the starting point of the SDE being null, 
as opposed to the non-zero case where the G\"artner-Ellis theorem applies directly (see~\cite{FJ09}).
In Section~\ref{sec:Results}, we present the model and state the large deviations results 
as time tends to zero; 
we also establish the connection with the Freidlin-Wentzell analysis via the contraction principle.
The proofs of the main results are gathered in Section~\ref{sec:Proofs}.\\
\textbf{Notations}: For a set $G\subset \RR^n$, 
we shall denote by $G^o$ and $\overline{G}$ its respective interior and closure in $\RR^n$.

\section{Main results}\label{sec:Results}

We consider here the following system of stochastic differential equations:
\begin{equation}\label{eq:Heston}
\begin{array}{rll}
\D X_t & = \displaystyle -\frac{1}{2}V_t\D t+ \sqrt{V_t}\D W_t, \quad & X_0=0,\\
\D V_t & = (a+b V_t)\D t+\xi\sqrt{V_t}\D Z_t, \quad & V_0=0,\\
\D\left\langle W,Z\right\rangle_t & = \rho \D t,
\end{array}
\end{equation}
where $a,\xi>0$, $b<0$, $|\rho|<1$ and $(W_t)_{t\geq0}$ and $(Z_t)_{t\geq0}$ are two standard Brownian motions.
We stress the importance of the parameter~$a$ to be strictly positive;
otherwise, the process~$V$, starting from zero, would just remain null,
and the unique solution of~\eqref{eq:Heston} would simply be the two-dimensional zero process.
We shall often make use of the notations $\bar{\rho}:=\sqrt{1-\rho^2}$ and $\mu:=2a/\xi^2$.
The SDE for the variance process $V$ has a unique strong solution 
by the Yamada-Watanabe conditions~\cite[Proposition 2.13, page 291]{KS97}. 
We further assume that $\mu > 1$, which ensures that the origin is unattainable for strictly positive times. 
Define now the following functions:
\begin{align} 
\Lambda_X^*(x) & = \left(u_- \ind_{\{x<0\}} + u_+\ind_{\{x\geq 0\}}\right) x, \label{eq:L*X}
\\
\Lambda_V^*(x) & = 
\left\{
\begin{array}{ll}
2x/\xi^2, & \text{if }x\geq 0,\\
+\infty, & \text{if }x<0.
\end{array}
\right. \label{eq:L*V}
\end{align}
for all $x\in\mathbb{R}$, where the two real numbers $u_-$ and $u_+$ read
\begin{equation}\label{eq:DefUpm}
\left.
\begin{array}{ll}
u_- & := \displaystyle \frac{2}{\xi\bar{\rho}}\arctan\left(\frac{\bar{\rho}}{\rho}\right)\ind_{\left\{\rho<0\right\}}
-\frac{\pi}{\xi}\ind_{\left\{\rho=0\right\}}
+\frac{2}{\xi\bar{\rho}}\left(\arctan\left(\frac{\bar{\rho}}{\rho}\right)-\pi\right)\ind_{\left\{\rho>0\right\}},
\\
u_+ & := \displaystyle 
\frac{2}{\xi\bar{\rho}}\arctan\left(\frac{\bar{\rho}}{\rho}\right)\ind_{\left\{\rho>0\right\}}
+\frac{\pi}{\xi}\ind_{\left\{\rho=0\right\}}
+\frac{2}{\xi\bar{\rho}}\left(\arctan\left(\frac{\bar{\rho}}{\rho}\right)+\pi\right)\ind_{\left\{\rho<0\right\}}.
\end{array}
\right.
\end{equation}
Note that $u_-$ (resp. $u_+$) is a decreasing (resp. decreasing) function of~$\rho$ and maps the interval $(-1,1)$ to $(-\infty,-2/\xi)$ (resp. $(2/\xi,+\infty)$).
We shall use the subscript/superscript $\MM$ to represent the quantities related to $X$ or to $V$. 
For instance $\Lambda^*_\MM$ represent $\Lambda^*_X$ or $\Lambda^*_V$.
We also denote $\mathcal{K}_{X}:=\mathbb{R}\setminus\{0\}$ and $\mathcal{K}_{V}:=(0,\infty)$.

\subsection{Large deviations results}
The main result of this paper is the following theorem, 
which provides an example of a sequence of random variables for which 
the limiting logarithmic cumulant generating function is zero (on its effective domain) but a large deviations principle still holds.
This is to be compared to the G\"artner-Ellis theorem~\cite[Theorem 2.3.6]{DZ93} which requires this limiting
function to be steep at the boundaries of its effective domain.
As highlighted in the proof, understanding the pointwise limit of the (rescaled) cumulant generating function
does not suffice any longer, and its higher-order behaviour is needed to prove large deviations.

\begin{theorem}\label{thm:LDP}
Whenever $\MM = X$ or $\MM = V$, the family $(\MM_{t})_{t\geq0}$ satisfies a LDP with speed~$t$ and rate function $\Lambda_\MM^*$ as $t$ tends to zero.
\end{theorem}

A more in-depth analysis reveals a more precise behaviour of the small-time probabilities, 
which take the following form as $t$ tends to zero:
\begin{equation}\label{eq:PPMx}
\PP(\MM_{t}\geq x)=
\left\{
\begin{array}{ll}
\displaystyle 1 - C(x)t^{1-\mu}\exp\left(-\frac{\Lambda_\MM^*(x)}{t}\right)
\left(1+\mathcal{O}(t)\right),
\qquad & \text{if } x<0,\\
\displaystyle C(x)t^{1-\mu}\exp\left(-\frac{\Lambda_\MM^*(x)}{t}\right)
\left(1+\mathcal{O}(t)\right),
\qquad & \text{if } x > 0,
\end{array}
\right.
\end{equation}
for all $x\in\mathcal{K}_{\MM}$, for some (smooth) strictly positive function~$C$.
This analysis is based on the so-called theory of sharp large deviations, 
developed in~\cite{Bercu, Bercu2},
and used in~\cite{Bercu3, JR, JR2} for diffusion processes and statistical estimators thereof.
It is based on refinements of the G\"artner-Ellis theorem in the case where the limiting cgf is not 
steep at the boundary; 
these refinements, using a time-dependent change of measure, were introduced in~\cite{BrycDembo}
and~\cite{DZ95}.
For in general $1-\mu \neq \frac{1}{2}$, we incidentally note that the factor $t^{1-\mu}$ 
here is not in line with the classical factor~$t^{1/2}$ found in heat kernel expansions for 
(the tail of the cumulative distribution function of) elliptic diffusions.

For the process~$V$, such an analysis, based on the moment generating function, is not really required, 
since the density of~$V_t$ is known explicitly for each~$t\geq 0$. 
Indeed, according to~\eqref{eq:cgfMg} below, its moment generating function reads
$$
\mathbb{E}\left(\E^{uV_t}\right) = \left(1+\frac{u\xi^2}{2b}\left(1-\E^{bt}\right)\right)^{-\mu}
 = \left(\frac{\lambda_t}{\lambda_t - u}\right)^{\mu},
$$
where $\lambda_t:=-\frac{2b}{\xi^2\left(1-\E^{bt}\right)}$, for all $u<\lambda_t$.
Hence~$V_t$ is distributed as a Gamma random variable with shape~$\mu$ 
and rate~$\lambda_t$.
Therefore, for any $x>0$, 
\begin{align*}
\PP(V_t \geq x)
 & = \frac{\lambda_t^\mu}{\Gamma(\mu)}\int_{x}^{+\infty}z^{\mu-1}\E^{-\lambda_t z}\D z
 = \frac{1}{\Gamma(\mu)}\int_{\lambda_t x}^{+\infty}z^{\mu-1}\E^{-z}\D z
 = \frac{(\lambda_t x)^{\mu-1}\E^{-\lambda_t x}}{\Gamma(\mu)}\left[1+\mathcal{O}\left(\frac{1}{\lambda_t x}\right)\right]\\
 & = \frac{\exp\left(\frac{bx}{\xi^2}\right)}{\Gamma(\mu)}
\left(\frac{2x}{\xi^{2}}\right)^{\mu-1}t^{1-\mu}\exp\left(-\frac{2x}{\xi^2 t}\right)\left[1+\mathcal{O}(t)\right],
\end{align*}
by the properties of the complementary incomplete Gamma function~\cite[Chapter 4, Section 2.1]{Olver},
and since $\lambda_t = \frac{2}{\xi^2 t} - \frac{b}{\xi^2} + \mathcal{O}(t)$;
therefore~\eqref{eq:PPMx} follows for~$\MM = V$.
In the proof of Theorem~\ref{thm:LDP} below, we will however keep the notation $\MM$,
standing for both $X$ and $V$, in order to highlight the fact that an analogous proof holds in both cases.

\subsection{Intuitions from Freidlin-Wentzell analysis}\label{sec:FW}
The proof of Theorem~\ref{thm:LDP} will be given in Section~\ref{sec:Proofs}.
In this section, we wish to illustrate how piece-wise linear rate functions such as~\eqref{eq:L*V} 
can arise from sample-path large deviations.
In order to simplify the framework, consider here the solution~$V$ 
of the equation obtained from~\eqref{eq:Heston} by setting $V_0 = v_0 >0$ and $a = b = 0$.
Setting $V^{\eps}_t := V_{\eps^2 t}$ for $\eps > 0$, the process $(V^{\eps}_t)_{t \ge 0}$ is the (weak) solution of the equation
\be \label{e:smallNoiseSimplified}
\D V^{\eps}_t  = \eps \xi\sqrt{V_t^\eps}\D Z_t, \qquad V^{\eps}_0=v_0 >0.
\ee
Pathwise large deviations (as $\eps$ tends to zero) for the solution of this SDE fall outside the scope 
of the classical Freidlin-Wentzell framework (as presented in~\cite[Chapter 5.6]{DZ93}) since the diffusion coefficient lacks the required 
global Lipschitz continuity property.
Based on properties of Bessel processes, Donati-Martin et al. in~\cite{DonatiYor} 
proved that, for every $T>0$, the process $(V^{\eps}_t)_{t \in [0,T]}$ does satisfy a large deviations principle 
on the path space $\Cc_T = C([0,T];\RR_+)$ of non-negative continuous functions,
with speed~$\eps^2$ and rate function~$I_T$ given by
\be \label{e:pathRateFct}
I_T(\varphi) = \left\{
\begin{array}{ll}
\displaystyle \frac{1}{2} \int_0^T \frac{\dot{\varphi}_t^2}{\xi^2 \varphi_t} \ind_{\{\varphi_t > 0\}} \D t,
& \text{if $\varphi \in \mathcal{C}_T$ is absolutely continuous and $\varphi_0 = v_0$}
\\
+\infty, & \text{otherwise},
\end{array} \right.
\ee
where one sets $y^{-1}\ind_{\{y > 0\}}=0$ when $y=0$.
More precisely, this means that the estimates
\be \label{e:pathLDP}
-\inf_{\varphi \in G^o} I_T(\varphi) \leq \liminf_{\eps \downarrow 0} \eps^2 \log \PP(V^{\eps} \in G)
\leq \limsup_{\eps \downarrow 0} \eps^2 \log \PP(V^{\eps} \in G) \leq -\inf_{\varphi \in \overline{G}} I_T(\varphi)
\ee
hold for every Borel set~$G\subset\Cc_T$.
By the contraction principle~\cite[Theorem 4.2.1]{DZ93}, the path estimates~\eqref{e:pathLDP} 
induce a LDP on~$\RR$ for the random variable $V^{\eps}_1 \equalDistrib V_{\eps^2}$, 
where the rate function is now given by
\be \label{e:rateFct}
\Lambda^*(x) := \inf\left\{I_1(\varphi): \varphi \in \Cc_1, \varphi_0=v_0, \varphi_1 = x\right\}.
\ee
This means that the sequence $(V_t)_{t\geq 0}$ 
satisfies a LDP with speed~$t$ as~$t$ tends to zero, namely for every Borel set~$A\subset\RR$,
$$
-\inf_{x \in A^o} \Lambda^*(x)
\leq \liminf_{t \downarrow 0} t \log \PP(V_t \in A)
\leq\limsup_{t \downarrow 0} t \log \PP(V_t \in A)
 \le -\inf_{x \in\overline{A}} \Lambda^*(x).
$$
\begin{proposition} \label{p:rateFcomput}
The rate function~$\Lambda^*$ in~\eqref{e:rateFct} reads
\[
\Lambda^*(x) =
\left\{ \begin{array}{ll}
\displaystyle \frac 2 {\xi^2} \left(\sqrt x - \sqrt{v_0}\right)^2, & \text{if } x\geq 0,\\
+\infty, & \text{if } x<0.
\end{array}
\right.
\]
\end{proposition}

For every $v_0>0$, $\Lambda^*$ is a strictly convex function on the positive real axis ($\partial_{xx}\Lambda^*(x)=\sqrt{v_0}/(2x^{3/2})$), 
converging pointwise to the affine function~$\Lambda^*_V(x) = 2 x / \xi^2$ 
given in~\eqref{eq:L*V} as $v_0$ tends to zero.
Note we are not claiming here (with the choice $a=0$ we made at the beginning of this section) 
that $\Lambda^*_V$ is the rate function for~$V^{\eps}_1$ in~\eqref{e:smallNoiseSimplified} when $v_0=0$: 
of course, in this case the unique solution to~\eqref{e:smallNoiseSimplified} is the identically null process $V^{\eps}\equiv 0$ 
(hence, the family $V^{\eps}_1$ satisfies a LDP with the trivial rate function~$\overline{I}(\varphi)=0$ 
if $\varphi\equiv 0$, and $\overline{I}(\varphi)=+\infty$ otherwise).

\begin{proof}[Proof of Proposition~\ref{p:rateFcomput}]
Let~$AC_+([0,1])$ denote the set of absolutely continuous functions on~$[0,1]$.
If~$x<0$, then by definition of~$I_1$, one has $I_1(\varphi)=+\infty$ for any~$\varphi$ such that~$\varphi_1=x$.
Then assume~$x\ge 0$, and consider~$\varphi \in \Cc_1$ such that~$\varphi_1=x$ and $I(\varphi)<+\infty$.
By the superposition principle (or the chain rule for absolutely continuous functions, 
see~\cite[Theorem 3.68]{Leoni}), the function $\psi \equiv \sqrt{\varphi}$ is absolutely continuous on every interval contained in the open set $\{\varphi>0\}$, 
with derivative almost surely equal to $\frac{\dot{\varphi}_t}{2\sqrt{\varphi_t}}\in L^2([0,1])$.
On $\{\varphi=0\}$ one has $\psi\equiv 0$, therefore $\dot{\psi}_t=0$ 
for every accumulation point of $\{\varphi=0\}$ (the isolated points form a finite subset of $[0,1]$).
In summary, it follows from~\cite[Corollary 3.26]{Leoni} that $\psi \in AC_+([0,1])$, and that 
$$
\int_0^1 \frac{\dot{\varphi}_t^2}{\varphi_t} \ind_{\{\varphi_t > 0\}} \D t 
= 4\int_{\{\varphi > 0\}} \dot{\psi}_t^2 \D t = 4\int_0^1 \dot{\psi}_t^2 \D t.
$$
Conversely, let $\psi \in AC_+([0,1])$ be such that $\dot{\psi}\in L^2([0,1])$, and set $\varphi \equiv \psi^2$; 
as the composition of a $C^1$ function and an absolutely continuous one, $\varphi$ also belongs to $AC_+([0,1])$ 
and $\dot{\varphi}_t = 2 \psi_t \dot{\psi}_t=2 \sqrt{\varphi_t} \dot{\psi}_t$ a.s.
Therefore, 
\be \label{e:chgV}
\begin{aligned}
\Lambda^*(x) =
 & \inf\left\{\frac1{2\xi^2} \int_0^1 \frac{\dot{\varphi}_t^2}{\varphi_t} \ind_{\{\varphi_t > 0\}} \D t
 : \varphi \in AC_+([0,1]) \mbox{ and } \varphi_0 = v_0, \varphi_1 = x\right\}
\\
 & = \inf\left\{\frac2{\xi^2} \int_0^1 \dot{\psi}_t^2 \D t: 
 \psi \in AC_+([0,1]) \mbox{ and } \psi_0 = \sqrt{v_0}, \psi_1 = \sqrt{x} \right\}.
\end{aligned} 
\ee
It is well known that the last problem is solved by the straight line 
$\psi^*_t \equiv \sqrt{v_0} + t(\sqrt x -\sqrt{v_0})$.
Substitution into~\eqref{e:chgV} yields 
$\Lambda^*(x)=\frac2{\xi^2} \int_0^1 (\dot{\psi}^*_t)^2=\frac2{\xi^2} (\sqrt x -\sqrt{v_0})^2$, 
and the proposition follows.
\end{proof}

Returning to the small-time problem for the solution~$V_t$ of~\eqref{eq:Heston}, now set~$V^{\eps}_t := V_{\eps^2 t}$, which satisfies
\be \label{e:smallTime}
\D V^{\eps}_t  = \eps^2(a + b V^{\eps}_t)\D t + \eps \xi\sqrt{V^{\eps}_t}\D Z_t, 
\qquad V^{\eps}_0=0.
\ee
To our knowledge, large deviations for the solution to~\eqref{e:smallTime} are not covered by the existing literature (in~\cite{DonatiYor}, 
the situation where the drift $a+bV$ is independent of~$\eps$ is considered).
We leave it to future research to prove that a pathwise LDP holds for the solution to~\eqref{e:smallTime} 
with a rate function analogous to~\eqref{e:pathRateFct}. Let us point out here that a pathwise LDP for the solution of the SDE
\[
\D \widetilde{V}^{\eps}_t = \eps^2 a(V^{\eps}_t)\D t + \eps \xi\sqrt{\widetilde{V}^{\eps}_t}\D Z_t
\]
starting at $\widetilde{V}^{\eps}_0=\eps^2 v_0$, was proven in~\cite[Theorem 1.1]{CDMD}.
Comparing with~\eqref{e:smallTime}, note that the initial condition is strictly positive, but tends to zero as~$\eps$ tends to zero; 
the condition $b=0$ is assumed, but the coefficient $a$ is allowed to be a bounded and Lipschitz continuous function with values in $\RR_+$.
\begin{remark}
The resulting rate function for the process $\widetilde V^{\eps}$ is identical to~$I_T(\varphi)$ given in~\eqref{e:pathRateFct}, 
albeit now with initial condition $\varphi_0 = 0$, for any choice of $a$ within this class.
Following analogous arguments to the proof of Proposition~\ref{p:rateFcomput}, 
the contraction principle applied to~$I_T$ yields exactly the rate function~\eqref{eq:L*V} 
for the family~$\widetilde{V}^{\eps}_1$, in line with what one expects from Theorem~\ref{thm:LDP}.
\end{remark}

\section{Proof of Theorem~\ref{thm:LDP}}\label{sec:Proofs}
The standard method to prove large deviations~\cite{DZ93} is to first prove an upper bound for the $\limsup$,
and then prove a lower bound for the $\liminf$, for the logarithmic probability of all Borel subsets of the real line.
We prove here directly that the limit holds for all open intervals of the form $(x,\infty)$ for $x\in \RR$,
which is clearly sufficient.
For any $t\geq 0$ and $\MM\in \{X, V\}$, define the rescaled cumulant generating function (cgf) $\Lambda_\MM(\cdot,t)$ 
of the random variable $\MM_t$ and its effective domain $\Dd_t^\MM$ by
\begin{equation}\label{eq:cgf}
\Lambda_\MM(u,t) := t \log\EE\left(\E^{u \MM_t/t}\right),\qquad
 \text{for all } u\in \Dd_t^\MM:=\left\{u\in\RR:|\Lambda_\MM(u,t)|<\infty\right\},\\
\end{equation}
Define further $\DdM:=\cap_{t>0}\Dd_t^\MM$.
From~\cite[Part I, Section 6.3.4]{JYCBook}, we know that 
\begin{equation}\label{eq:cgfMg}
\begin{array}{rlrl}
\Lambda_\MM(u,t) & = 
\displaystyle -\frac{\mu t }{2}\left[g_t^{\MM}(u) + 2\log f_t^{\MM}(u)\right],
\qquad  & \displaystyle f_t^{X}(u) & \equiv \displaystyle 
\cosh\left(\frac{d(\frac{u}{t})t}{2}\right)-\frac{ g_t^{X}(u)}{td(\frac{u}{t})}\sinh\left(\frac{d(\frac{u}{t})t}{2}\right), \\
f_t^{V}(u) & \equiv \displaystyle 1+\frac{u\xi^2}{2b t}\left(1-\E^{bt}\right),
\qquad  & \displaystyle g_t^{X}(u) & \equiv bt+\rho\xi u,
\qquad g_t^{V}(u) \equiv 0,
\end{array}
\end{equation}
where 
$d(u) \equiv [(b + \rho\xi u)^2+u(1-u)\xi^2]^{1/2}$, 
so that the functions $\Lambda_X(\cdot, t)$ and $\Lambda_V(\cdot, t)$ are 
explicitly well defined on $\mathcal{D}^X_{t}$ and $\mathcal{D}^V_{t}$.
The pointwise limit functions
$\Lambda_\MM(u):=\lim_{t\downarrow 0}\Lambda_\MM(u,t)$, for $\MM\in\{X,V\}$,
read as follows:
\begin{lemma} \label{lemma:LimitCGF}\text{}
The function $\Lambda_\MM$ is null on $\Dd_\MM$ and infinite outside, 
with $\Dd_X = (u_-,u_+)$ and $\Dd_V = (-\infty,2/\xi^2)$.
\end{lemma}
\begin{proof}
The lemma follows from a simple yet careful analysis of the functions~$\Lambda_X(\cdot, t)$ 
and~$\Lambda_V(\cdot, t)$ together with their effective domains.
Clearly here $\mathcal{D}^V_{t} = (-\infty, u^V(t))$, where
$u^V(t)\equiv 2bt/[\xi^2(\E^{bt}-1)]$ converges from above to $2/\xi^2$ as $t$ tends to zero.
In~\cite{FJ09}, the authors showed that $u_+(t)$ (resp. $u_-(t))$ converges from above 
(resp. from below) to $u_+$ (resp. $u_-$) as $t$ tends to zero, so that the limiting domain $\cap_{t>0}\mathcal{D}^X_t$ is equal 
to $(u_-,u_+)$.
The pointwise limits are then straightforward to prove.
\end{proof}

Define now the following functions on $\DdM^{o}$:
\begin{equation}\label{eq:fgasymp}
\left\{
\begin{array}{rl}
\ff_0^{X}(u) & := \displaystyle \cos(\bar{\rho}\xi u/2) - \frac{\rho}{\bar{\rho}} \sin(\bar{\rho}\xi u/2),
\quad
\ff_0^{V}(u):=1-\frac{u\xi^2}{2},
\qquad
\ff_1^{V}(u):=-\frac{b u \xi^2}{4},
\\
\ff_1^{X}(u) & := \displaystyle \left\{
\begin{array}{ll}
\displaystyle \frac{\rho(\xi+2b\rho)}{4 \bar{\rho}^2}\cos(\bar{\rho}\xi u/2) 
+ \left(\frac{\xi+2b\rho}{4\bar{\rho}}-\frac{\xi\rho+2b}{2u\xi\bar{\rho}^3}\right) \sin(\bar{\rho}\xi u/2), & \text{if }u\neq0,\\
-b/2, & \text{if }u=0,
\end{array}
\right. , \\
\g_0^{X}(u) & := \rho\xi u,
\qquad 
\g_0^{V}(u)\equiv 0.
\end{array}
\right.
\end{equation}

\begin{lemma}\label{lemma:fg}
For $\MM\in\{X,V\}$, the expansions 
$f_t^{\MM}(u) = \ff_0^{\MM}(u) + \ff_1^{\MM}(u)t+\mathcal{O}(t^2)$
and 
$g_t^{\MM}(u) = \g_0^{\MM}(u) + \mathcal{O}(t)$
hold for all $u\in\DdM^{o}$ as $t$ tends to zero.
They further hold uniformly on compacts.
\end{lemma}
\begin{proof}
Let $\MM=X$, and define the quantities
$d_0:= \bar{\rho}\xi\sgn(u)$,
$d_1:=\frac{ \I\left(2 \kappa  \rho -\xi \right)\sgn(u)}{2 \bar{\rho}}$,
where $\sgn(u)=1$ if $u\geq 0$, and~$-1$ otherwise;
then for any $u\in\mathcal{D}_{X}^{o}\backslash\{0\}$,
$d(u/t) = \I u d_0 / t + d_1 + \mathcal{O}(t)$,
as $t$ tends to zero, and hence 
\begin{equation}\label{eq:fXasymp}
\begin{array}{rl}
\displaystyle \frac{g_t^{X}(u)}{td(u/t)} &\displaystyle 
 =\frac{\rho\xi u}{\I u d_0}+\frac{d_1\xi\rho-\I b d_0}{d_0^2 u}t+\mathcal{O}(t^2), \\
\displaystyle \cosh\left(\frac{d(u/t)t}{2}\right) & \displaystyle 
 =\cos\left(\frac{d_0 u}{2}\right)+\frac{\I d_1}{2}\sin\left(\frac{d_0 u}{2}\right)t +\mathcal{O}(t^2), \\
\displaystyle \sinh\left(\frac{d(u/t)t}{2}\right) &\displaystyle 
 =\I \sin\left(\frac{d_0 u}{2}\right)+\frac{d_1}{2}\cos\left(\frac{d_0 u}{2}\right)t +\mathcal{O}(t^2).
\end{array}
\end{equation}
The expansion for $f_t^{X}$ in~\eqref{eq:cgfMg} for $u\in\mathcal{D}_{X}^{o}\backslash\{0\}$ 
follows after using the asymptotics in~\eqref{eq:fXasymp} and some simplification.
When $u=0$, straightforward computations reveal that $f_t^{X}(u)=1-bt/2+\mathcal{O}(t^2)$, 
in agreement with~\eqref{eq:cgfMg}.
Note that $f_1^X$ is continuous at the origin.
The expansions for $f_t^{V}$ and $g_t^{\MM}$ follow analogous arguments and the lemma follows. 
Uniform convergence of the sequences $(g_t^\MM)_t$ and $(f_t^V)_t$ is trivial.
Uniform convergence on compacts of the sequence $(f^X_t)_t$ holds as soon as 
$\sup_{u\in\Dd^o_X}\left|f_t^{\MM}(u)-f_0^{\MM}(u)-f_1^{\MM}(u)t\right|$ converges to zero when~$t$ tends to zero,
which is tedious but straightforward to prove.
\end{proof}

Consider now the (time-dependent) saddlepoint equation: 
\begin{equation}\label{eq:saddlepoint}
\partial_{u}\Lambda_{\MM} (u,t)=x,\qquad 
x\in\mathcal{K}_{\MM}, t>0.
\end{equation}
The following lemma proves existence and uniqueness of the solution to this equation, 
as well as a small-time expansion.
Let us first define the following functions on $\mathcal{K}_{\MM}$:
\begin{equation}\label{eq:calpha}
\alpha_0^X(x) := \displaystyle u_- \ind_{\{x<0\}} + u_+\ind_{\{x> 0\}},
\qquad
\qquad
\alpha_0^V(x):=\frac{2}{\xi^2} \ind_{\{x\geq 0\}},
\end{equation}

\begin{lemma}\label{lemma:u^*smalltime}
For any $x\in\mathcal{K}_\MM$, $t>0$, Equation~\eqref{eq:saddlepoint} admits a unique solution~$u^*_{\MM}(x,t) \in \Dd_t^\MM$; moreover, the expansion $u^*_\MM(x,t) = \alpha_0^\MM(x) + \mathcal{O}(t)$ holds.
For $x=0$, Equation~\eqref{eq:saddlepoint} also admits a unique solution~$u^*_{\MM}(0,t)$, 
which converges to zero as $t$ approaches zero.
\end{lemma}

\begin{proof}
We first prove existence and uniqueness of the solution of the saddlepoint equation~\eqref{eq:saddlepoint}.
Consider first the case $\MM=V$. 
Clearly, for any $t>0$, the map $\partial_u \Lambda_{V}(\cdot,t):\mathcal{D}_{t}^{V}\to\mathbb{R}$ is strictly increasing
and the image of 
$\mathcal{D}^V_{t}$ by $\partial_u \Lambda_{V}(\cdot,t)$ is $\mathbb{R}^{*}_{+}$.
Thus, for any $x>0$,~\eqref{eq:saddlepoint} admits a unique solution
$u_{V}^*(x,t)=
\frac{2 t }{ \xi ^2 }\left(\frac{b}{\E^{bt}-1}-\frac{a}{x}\right)
$, 
which converges to $\alpha_0^V(x)$.
Consider now  the case $\MM=X$. 
We first start with the following claims, which can be proved using the convexity of the moment generating function and tedious computations.
\begin{enumerate}[(i)]
\item For any $t>0$, the function $\partial_u \Lambda_{X}(\cdot,t):\mathcal{D}_{t}^{X}\to\mathbb{R}$ is strictly increasing
and maps 
$\mathcal{D}^X_{t}$ to $\mathbb{R}$;
\item For any $t>0$, $u_{X}^*(0,t)>0$  and $\lim_{t\downarrow 0}u_{X}^*(0,t)=0$,
i.e. the unique minimum of $\Lambda_{X}(\cdot,t)$ converges to zero;
\item For each $u\in\mathcal{D}_X^{o}$, $\partial_u \Lambda_{X}(u,t)$ converges to zero 
as $t$ tends to zero.
\end{enumerate}
Now, choose $x>0$ (analogous arguments hold for $x<0$).
It is clear from~(i) that~\eqref{eq:saddlepoint} admits a unique solution.
Note further that (i) and (ii) imply $u_{X}^*(x,t)>0$.
Next we introduce the following condition.
\begin{equation*}
\textbf{Condition A}: \text{There exists $t_1>0$ such that $u^*_{X}(x,t)\in\mathcal{D}_X^{o}$ for all $t<t_1$}.
\end{equation*}
Suppose condition A is not true and further assume that the sequence $(u_{X}^*(x,t))_{t>0}$ does not converge to $u_{+}^{X}$ as~$t$ tends to zero.
Then there exists $t^*_1>0$ and $\eps>0$ such that for all $t<t^*_1$ we have $u_{X}^*(x,t)\not\in B(u_{+}^{X},\eps):=\{y\in\mathbb{R}:|y-u_{+}^{X}|<\eps\}$.
But since $\lim_{t\downarrow0}\mathcal{D}_t^{X}=\mathcal{D}_{X}$, this implies that our sequence must then satisfy condition A, which is a contradiction. 
Therefore $u_{X}^*(x,t)$ converges to $u_{+}^{X}$.
Next suppose that condition A is true.
Again note that (i) and (ii) imply $u_{X}^*(x,t)>0$.
From~(iii) there exists $t_2>0$ such that the sequence $(u_{X}^*(x,t))_{t>0}$ 
is strictly increasing as $t$ goes to zero for $t<t_2$.
Now let $t^*=\min(t_1,t_2)$ and consider $t<t^*$.
Then $u_{X}^*(x,t)$ is bounded above by $u_+$ (because $u_{X}^*(x,t)\in\mathcal{D}_X^{o}$) 
and therefore converges to a limit $L\in[0,u_+]$. 
Suppose that  $L\ne u_+$. 
Since $s\mapsto u_X^*(x,s)$ is strictly increasing as $s$ tends to zero (and $s<t^*$), 
and $\partial_u \Lambda_{X}(\cdot,t)$ is strictly increasing we have $\partial_u \Lambda_{X}(u_X^*(x,t),t)\leq \partial_u \Lambda_{X}(L,t)$;
Combining this and~(iii) yields
$
\lim_{t\downarrow 0}\partial_u \Lambda_{X}(u_X^*(x,t),t) \leq  
\lim_{t\downarrow 0}\partial_u \Lambda_{X}(L,t) = 0
\neq x,
$
which contradicts the assumption $x>0$. 
Therefore $L=u_+$ and the first part of the lemma follows.

Given existence and uniqueness of the solution to the saddlepoint equation, 
we now prove the expansion stated in the lemma.
In light of~\eqref{eq:cgfMg}, the saddlepoint equation~\eqref{eq:saddlepoint} can be written explicitly as
$$
-\frac{\mu t}{2}\left[
\partial_u g_t^{\MM}(u_\MM^*(x,t))  f_t^{\MM}(u_\MM^*(x,t)) + 2 \partial_u f_t^{\MM}(u_\MM^*(x,t))
\right] = f_t^{\MM}(u_\MM^*(x,t)) x .
$$
Using Lemma~\ref{lemma:fg} in this equation and solving at each order yields the desired expansion.
\end{proof}

For $\MM\in\{X,V\}$ and $t>0$, introduce now a time-dependent change of measure by
\begin{equation}\label{eq:measureChange}
\frac{\D\mathbb{Q}^\MM_{x,t}}{\D\PP} := 
\exp\left(\frac{u^*_{\MM}(x,t)\MM_{t}}{t} - \frac{\Lambda_{\MM}(u^*_{\MM}(x,t),t)}{t}\right).
\end{equation}
By Lemma~\ref{lemma:u^*smalltime}, $u^*_\MM(x,t)$ belongs to the interior of $\mathcal{D}_t^{\MM}$,
and so $|\Lambda_{\MM}(u^*_\MM(x,t))|$ is finite.
Also $\D\mathbb{Q}^\MM_{x,t}/\D\PP$ is almost surely strictly positive and $\EE[\D\mathbb{Q}^\MM_{x,t}/\D\PP]=1$. 
Therefore~\eqref{eq:measureChange} is a valid measure change for all $t>0$ and $x\in\mathcal{K}_{\MM}$.

\noindent
Define now the random variable $\label{eq:DefZ} \ZZ:=(\MM_{t}- x)$, 
and denote its  characteristic function in the $\mathbb{Q}^\MM_{x,t}$-measure~\eqref{eq:measureChange}
by $\Phi^\MM_{x,t}(u):=\EE^{\mathbb{Q}^\MM_{x,t}}(\E^{\I u \ZZ})$.
We apply Fourier inversion methods to derive the tail probabilities of $\ZZ$ under the measure~\eqref{eq:measureChange}. 
It follows from Lemmas \ref{lem:invfourtranszero} and \ref{lem:L1lemZero} that, for any $x\in\mathcal{K}_{\MM}$, the following estimate holds as $t$ tends to zero:
\begin{equation} \label{e:fourierasympsmalltime}
\EE^{\QQ}\left[\exp\left(-\frac{u^*_{\MM}(x,t)\ZZ}{t}\right)\ind_{\{\ZZ\geq 0\}}\right]\ind_{\{x>0\}}
=\mathcal{O}(t).
\ee

Finally, we have the following additional estimate.
\begin{lemma} \label{Lemma:U*Asymptotics}
For any $x\in\mathcal{K}_{\MM}$, there exists a constant $C>0$ such that the expansion
$$
\exp\left[-\frac{xu^*_{\MM}(x,t)}{t}+\frac{\Lambda_{\MM}(u^*_\MM(x,t),t)}{t}\right]
=
C \exp\left(-\frac{\Lambda_\MM^*(x)}{t}\right)
\mathcal{O}\left(t^{-\mu}\right)
$$
holds as $t$ tends to zero, with $\Lambda_\MM^*$ in~\eqref{eq:L*X},\eqref{eq:L*V}.
\end{lemma}
\begin{remark}
The constant $C$ above can be computed explicitly as 
$$
C = \left(-\frac{\mu\partial_u f_0^\MM(\alpha_0^\MM(x))}{x}\right)^{-\mu}
\exp\left(-x\alpha_1^\MM(x)-\frac{\mu}{2}g_0^{\MM}(\alpha_0^\MM(x))\right).
$$
\end{remark}
\begin{proof}
From Lemma~\ref{lemma:u^*smalltime} and the characterisation of $\Lambda^*$ 
in~\eqref{eq:L*X},\eqref{eq:L*V}, 
there exists $C>0$, such that for small~$t$,
\begin{equation}\label{eq:asymp1}
\exp\left(-\frac{xu^*_{\MM}(x,t)}{t}\right)
 = C\exp\left(-\frac{x\alpha_0^\MM(x)}{t}\right)(1+\mathcal{O}(t))
 = C\exp\left(-\frac{\Lambda_\MM^*(x)}{t}\right)(1+\mathcal{O}(t)).
\end{equation}
The definition of $\Lambda_{\MM}$ in~\eqref{eq:cgfMg} and Lemma~\ref{lemma:u^*smalltime}
yield $\E^{\Lambda_{\MM}(u^*_\MM(x,t),t)/t} = \mathcal{O}\left(t^{-\mu}\right)$,
and the lemma follows.
\end{proof}

We now put all the pieces together.
Using the time-dependent change of measure~\eqref{eq:measureChange}, we have for $x>0$,
\begin{align*}
\PP(\MM_{t}\geq x)
 = \EE\left[\ind_{\{\MM_{t}\geq x\}}\right] & =
\exp\left(\frac{\Lambda_{\MM}(u^*_{\MM}(x,t)}{t}\right)
\EE^{\mathbb{Q}^\MM_{x,t}}\left[\exp\left(-\frac{u^*_{\MM}(x,t)\MM_{t}}{t}\right)
\ind_{\{\MM_{t}\geq x\}}\right] \\
 & = \exp\left(-\frac{xu^*_{\MM}(x,t) - \Lambda_{\MM}(u^*_{\MM}(x,t),t)}{t}\right)
\EE^{\QQ}\left[\exp\left(-\frac{u^*_{\MM}(x,t)\ZZ}{t}\right)\ind_{\{\ZZ\geq 0\}}\right],
\end{align*}
with $\ZZ$ defined on page~\pageref{eq:DefZ}. 
An analogous argument holds for probabilities $\PP(\MM_{t}\leq x)$ when $x<0$, 
and Theorem~\ref{thm:LDP} the follows from Lemma~\ref{Lemma:U*Asymptotics} and Equation\eqref{e:fourierasympsmalltime}.

\appendix
\section{The G\"artner Ellis Theorem}\label{app:LDP}
We provide here a brief review of large deviations and the G\"artner-Ellis theorem.
For a detailed account of these, the interested reader should consult~\cite{DZ93}.
Let $(X_n)_{n\in\NN}$ be a sequence of random variables in~$\RR$, with law~$\mu_n$
and cumulant generating function $\Lambda_n(u) \equiv \log\EE(\E^{u X_n})$.

\begin{definition}
The sequence {$X_n$} is said to satisfy a large deviations principle with speed $n$ 
and rate function~$I$ if for each Borel measurable set $E\subset\RR$, 
$$
-\inf_{x\in E^o} I(x)
 \leq \liminf_{n\uparrow \infty} \frac{1}{n}\log\PP\left(X_n \in E \right)
\leq \limsup_{n\uparrow \infty} \frac{1}{n}\log\PP\left(X_n \in E \right)
\leq -\inf_{x\in \bar{E}}I(x).
$$

\end{definition}
Before stating the main theorem, we need one more concept:
\begin{definition}\label{def:EssSmooth}
Let $\Lambda : \RR \rightarrow (-\infty, +\infty]$ be a convex function, 
and $\Dd_\Lambda := \{u\in\RR: \Lambda(u) < \infty\}$ its effective domain. 
It is said to be essentially smooth if
\begin{itemize}
\item The interior $\Dd^o_\Lambda$ is non-empty;
\item $\Lambda$ is differentiable throughout $\Dd^o_\Lambda$;
\item $\Lambda$ is steep: $\lim \limits_{n\uparrow \infty}|\Lambda' (u_n)|=\infty$ 
whenever $(u_n)$ is a sequence in $\Dd_\Lambda^o$ converging to a boundary point 
of~$\Dd_\Lambda^o$.
\end{itemize}
\end{definition}
Assume now that the limiting cumulant generating function
$\Lambda(u) := \lim_{n\uparrow \infty }n^{-1}\Lambda_n(nu)$,
exists as an extended real number for all $u\in\RR$, and let $\Dd_\Lambda$ denote its effective domain.
Let $\Lambda^*:\RR\to\RR_+$ denote its (dual) Fenchel-Legendre transform, via the variational formula
$\Lambda^*(x) \equiv \sup_{u \in D_\Lambda}\{u x-\Lambda(u)\}$.
Then the following holds:
\begin{theorem}[G\"artner-Ellis theorem]\label{thm:GE}
If the origin lies in the interior of~$\Dd_{\Lambda}$ and if $\Lambda$ is lower semicontinuous and essentially smooth,
then the sequence $(X_n)_n$ satisfies a large deviations principle with rate function~$\Lambda^*$.
\end{theorem}

\section{Inverse Fourier Transform Representation}

Let $g(z):=\exp\left(-u^*_{\MM}(x,t)z/t\right)\ind_{\{z\geq 0\}}$. 
The main result of this appendix is the following representation:
\begin{lemma}\label{lem:invfourtranszero}
For every $x>0$, there exists $t^*_1>0$ such that for all $t<t^*_1$:
\begin{equation}\label{eq:ParsZero}
\EE^{\QQ}\left[g(Z^{\MM}_{x,t})\right]
 = \frac{1}{2\pi}\int_{\RR}\Phi^\MM_{x,t}(u)\left(\frac{u^*_{\MM}(x,t)}{t}+\I u\right)^{-1}\D u.
\end{equation}
\end{lemma}
The proof of Lemma~\ref{lem:invfourtranszero} proceeds in two steps:
We first prove that the integrand in the right-hand side of Equality~\eqref{eq:ParsZero} belongs to~$L^1(\RR)$ 
(and hence the integral is well-defined), 
and then prove that this very equality holds. 
The first step is contained in the following lemma.
\begin{lemma}\label{lem:L1lemZero}
We have $\int_{\RR}\left|\frac{\Phi^{\MM}_{x,t}(u)}{u^*_{\MM}(x,t)/t + \I u}\right| \D u = \mathcal{O}(t)$ as $t$ tends to zero.
In particular, there exists $t^*_0>0$ such that
$\int_{\RR}\left|\frac{\Phi^{\MM}_{x,t}(u)}{u^*_{\MM}(x,t)/t + \I u}\right| \D u$ is finite
 for all $t<t^*_0$.
\end{lemma}
\begin{proof}
From the change of measure~\eqref{eq:measureChange} and the re-scaled cgf
given in~\eqref{eq:cgf} we can compute
$$
\log\Phi^\MM_{x,t}(u)
  = \log\mathbb{E}^{\mathbb{P}}\left(\frac{\D\mathbb{Q}^\MM_{x,t}}{\D\PP} \exp\left(\I u  \ZZ\right)\right)
 = -\I u x+\frac{1}{t}\Big[\Lambda_{\MM}(\I u t +u^*_{\MM}(x,t),t)-\Lambda_{\MM}(u^*_{\MM}(x,t),t)\Big].
$$
Using the definition of $\Lambda_{\MM}$ in~\eqref{eq:cgfMg} then yields
\begin{equation}\label{eq:PhiMExplicit}
\Phi^\MM_{x,t}(u)=
\left(\frac{f_t^{\MM}(\I u t+u_\MM^*(x,t))}{f_t^{\MM}(u_\MM^*(x,t))}\right)^{-\mu}
\exp\left(-\I u x -\frac{2}{\mu}\left[g_t^{\MM}(u_\MM^*(x,t)+\I u t) -g_t^{\MM}(u_\MM^*(x,t)) \right]\right).
\end{equation}
Let now $\MM=V$.  Using~\eqref{eq:PhiMExplicit} we see that $\Phi^V_{x,t}(u)=\E^{-\I u x}\left(1+\I u d_t \right)^{-\mu}$ where
$d_t:=\frac{\xi^2 t (1-\E^{b t})}{2 b t+u^*_{V}(x,t)\xi^2(1-\E^{bt})}$.
It is easy to see that $\lim_{t\downarrow0}d_t = b\xi^2$.
The modulus is then given by $|\Phi^V_{x,t}(u)|=(1+u^2d_t^2)^{-\mu/2}$
and hence for small enough $t$ we have that $|\Phi^V_{x,t}(u)|\leq D |u|^{-\mu}$ for some $D>0$.
Furthermore, we easily see that
$|(u^*_{V}(x,t)+\I u t)^{-1}|\leq 1/|u^*_{V}(x,t)|$ for all $t$, and hence we compute
\begin{align*}
\int_{\RR}\left|\frac{\Phi^V_{x,t}(u)}{u^*_{V}(x,t)/t+\I u}\right| \D u
=
t \int_{\RR}\left|\frac{\Phi^V_{x,t}(u)}{u^*_{V}(x,t)+\I u t}\right| \D u
&=
t \int_{|u|\leq1}\left|\frac{\Phi^V_{x,t}(u)}{u^*_{V}(x,t)+ \I u t}\right| \D u
+ t \int_{|u|>1}\left|\frac{\Phi^V_{x,t}(u)}{u^*_{V}(x,t)+\I u t}\right| \D u
\\
&\leq t \: \frac{1}{u^*_{V}(x,t)} \Bigl(1 + D \int_{|u|>1} \frac{\D u}{|u|^{\mu}} \Bigr).
\end{align*}
The factor multiplying $t$ in the last inequality is bounded for sufficiently small $t$, since $u^*_{V}(x,t)$ converges to~$2/\xi^2$ as $t$ tends to zero, and $\mu > 1$.
The case $\MM=X$ follows from analogous, yet tedious, computations.
\end{proof}

We now move on to the proof of Lemma~\ref{lem:invfourtranszero}.
We only look at the case  $\MM=V$, the other case being completely analogous.
We denote the convolution of two functions $f,h\in L^1(\RR)$ by $(f\ast h)(x):=\int_{\RR}f(x-y)h(y) \D y$,
and recall that $(f\ast h)\in L^1(\RR)$.
For such functions, we denote the Fourier transform by
$(\mathcal{F}f)(u):=\int_{\RR}\E^{\I u x}f(x) \D x$ 
and the inverse Fourier transform by $
(\mathcal{F}^{-1}h)(x):=\frac{1}{2\pi}\int_{\RR}\E^{-\I u x}h(u) \D u.
$ 
We have that
\begin{equation}\label{eq:stripszero}
\mathcal{F}\left(g(z)\right)(u)
 := \int_{\RR}\exp\left(-\frac{u^*_{V}(x,t)z}{t}+\I u z\right)\ind_{\{z\geq 0\}}\D z
 = \left(\frac{u^*_{V}(x,t)}{t}-\I u\right)^{-1},
\end{equation}
if $u^*_{V}(x,t)>0$, which holds for small~$t$ since by Lemma~\ref{lemma:u^*smalltime} $u^*_{V}(x,t)$ converges to $\alpha_{0}^{V}$ and $\alpha_{0}^{V}>0$.
We write
$$
\EE^{\mathbb{Q}^{V}_{x,t}}\left[g(Z^{V}_{x,t})\right]
=\int_{\RR}q(x-y)p(y) \D y = (q\ast p)(x),
$$
with $q(z)\equiv g(-z)$ and $p$ denoting the density of $V_t$.
On the strips of regularity ($x>0$) we know there exists $t_0>0$ such that
$q\in L^1(\RR)$ for $t<t_0$.
Since $p$ is a density, $p\in L^1(\RR)$, and therefore 
\begin{equation}\label{eq:convzero}
\mathcal{F}(q\ast p)(u)=\mathcal{F}q(u) \mathcal{F}p(u).
\end{equation}
We note that $\mathcal{F}q(u)\equiv\mathcal{F}g(-u)\equiv\overline{\mathcal{F}g(u)}$ and hence using~\eqref{eq:stripszero}
\begin{equation}\label{eq:simpconvzero}
\mathcal{F}q(u) \mathcal{F}p(u)
\equiv\E^{\I u x}\Phi_{x,t}^{V}(u)\left(\frac{u^*_{V}(x,t)}{t}+\I u\right)^{-1},
\end{equation}
since the complex conjugate of $w^{-1}$ is equal to $(\Re(w)-\I\Im(w))^{-1}$, for $w\in\mathbb{C}$. 
Thus by Lemma~\ref{lem:L1lemZero} there exists an $t_1>0$ 
such that $\mathcal{F}q \mathcal{F}p\in L^1(\RR)$ for $t<t_1$. 
By the inversion theorem~\cite[Theorem 9.11]{R87} this then implies from~\eqref{eq:convzero} and~\eqref{eq:simpconvzero} that for $t<\min(t_0,t_1)$:
\begin{align*}
\EE^{\mathbb{Q}^{V}_{x,t}}\left[g(Z^{V}_{x,t})\right]
&= (q\ast p)(x) 
=\mathcal{F}^{-1}\left(\mathcal{F}q(u) \mathcal{F}p(u)\right)(x) \\ 
&=\frac{1}{2\pi}\int_{\RR} \E^{-\I u x}\mathcal{F}q(u) \mathcal{F}p(u) \D u
= \frac{1}{2\pi}\int_{\RR} \Phi_{x,t}^{V}(u)\left(\frac{u^*_{V}(x,t)}{t}+\I u\right)^{-1}\D u .
\end{align*}


\end{document}